\documentclass[12pt]{amsart}

\usepackage{amssymb,amsfonts,multicol}

\usepackage{palatino}
\usepackage[mathscr]{eucal}

\allowdisplaybreaks

\newcommand{\oN}{\mathbb{N}} 
\newcommand{\oZ}{\mathbb{Z}}
\newcommand{\oC}{\mathbb{C}} 
\newcommand{\oCP}{\mathbb{C}\kern-1pt P}

\newcommand{\q}{\mathfrak{q}}
\newcommand{\z}{\mathbf{z}}
\newcommand{\dd}{{\partial\kern-7.4pt\partial\kern-7.6pt\partial\kern2pt}}

\newcommand{\W}{\mathsf{W}}
\newcommand{\theg}{\mathfrak{g}}
\newcommand{\thega}{\hat{\mathfrak{g}}}
\newcommand{\theh}{\mathfrak{h}}
\newcommand{\theG}{\bf{G}}
\newcommand{\theP}{\bf{P}}
\newcommand{\theF}{\bf{F}}
\newcommand{\theb}{\mathfrak{b}}
\newcommand{\theB}{\bf{B}}
\newcommand{\theH}{\bf{H}}
\newcommand{\theX}{\bf{X}}
\newcommand{\sroots}{\mathbf{\Pi}}
\newcommand{\roots}{\mathbf{\Delta}}
\newcommand{\voa}[1][X]{\mathcal{L}\!{#1}}
\newcommand{\voahr}[1][X]{\mathcal{W}\!{#1}}
\newcommand{\voaw}[1][X]{\mathcal{WL}\!{#1}}

\newcommand{\modV}{\mathfrak{V}}
\newcommand{\modX}{\mathfrak{X}}
\newcommand{\modF}{\mathfrak{F}}
\newcommand{\modM}{\mathcal{M}}

\newcommand{\oneb}{\mathbf{1}\kern-5pt\mathbf{1}}

\newcommand{\ad}{{\rm ad}}
\newcommand{\Tr}{{\rm Tr}}

\newtheorem{Thm}{Theorem}[section]
\newtheorem{Lemma}[Thm]{Lemma}
\newtheorem{Prop}[Thm]{Proposition}

\theoremstyle{definition}
\newtheorem{Dfn}[Thm]{Definition}
\newtheorem{Rem}[Thm]{Remark}

\newcommand{\scalar}[2]{(#1\,,\,#2)}
\newcommand{\commut}[2]{[#1\,,\,#2]}
\newcommand{\qad}{{\mathsf{qad}}}

\newcommand{\half}{\frac{1}{2}}

\begin{document}

\title[Logarithmic CFTs and root systems]{%
  Logarithmic CFTs connected with simple Lie algebras\\
  {\tiny Simply-laced case}}

 \author{B.L.~Feigin, and I.Yu.~Tipunin}
 \address{BLF:Higher School of Economics, Moscow, Russia and Landau institute for Theoretical Physics,
 Chernogolovka, 142432, Russia}
 \email{bfeigin@gmail.com}
 \address{IYuT:Tamm Theory Division, Lebedev Physics Institute, Leninski pr., 53,
 Moscow, Russia, 119991}
 \email{tipunin@gmail.com}

\begin{abstract}
 For any root system corresponding to a semisimple simply-laced Lie algebra 
a logarithmic CFT is constructed. Characters of irreducible representations were 
calculated in terms of theta functions.
\end{abstract}

\maketitle

\section{Introduction}
The most deeply investigated examples of LCFTs are $(p,p')$ models introduced in~\cite{PRZ}.
The chiral algebra of these models~\cite{FGST3} is based on the $s\ell(2)$ symmetry discovered 
in~\cite{Kausch}. In the paper, we generalize a construction of a chiral algebra to the case
of a semisimple simply laced Lie algebra.

We suggest a method to construct LCFTs from the following data. Let $\theB$ be a Lie
group, $\theX$ be a space with an action of $\theB$ on it and let $\mathcal{L}$ be a vertex operator algebra
in which $\theB$ acts by symmetries of OPEs. 
We construct a bundle $\xi=\mathcal{L}\otimes_{\theB}\theX$ with fibers $\mathcal{L}$.
Then the cohomology of $\xi$ bear a VOA structure. This VOAs are main object
of investigation in this paper. We assume notation $\mathcal{W}$ for this VOA. On this way we obtain VOAs 
that generalize $W$-algebras of $(1,p)$ 
models~\cite{FHST,GabKausch,Flohr1}. Generalization of $(p,p')$ models from~\cite{FGST3}
also can be obtained on this way.

In the paper we consider the case where $\theB$ is the Borel subgroup in a simply laced Lie
group~$\theG$, the space $\theX$ is the group~$\theG$ itself and $\mathcal{L}$ is a lattice VOA
in which $\theB$ acts by screening operators. In this case only zero-dimension cohomologies
of the sheaf $\xi$ are nonzero and therefore we can identify global sections of $\xi$
with a vertex operator algebra~$\mathcal{W}$. 
Irreducible modules of~$\mathcal{W}$ can be constructed as cohomology of some bundles
on~$\theG/\theB$. We consider the 
sheaf of sections of the bundle $\xi(\modV)=\modV\otimes_{\theB}\theG$, where
$\modV$ is an irreducible $\mathcal{L}$-module.  We obtain irreducible modules 
of $\mathcal{W}$ as global sections of the bundle~$\xi(\modV)$. 

Most of statements in the paper can be proved using results of~\cite{Arakawa}
on W-algebras.

In the case where ${\theG}=SL(2)$ the algebra $\mathcal{W}$ is described as a kernel
of the screening~\cite{FHST} and its irreducible modules are described in the similar
way. This allows us to calculate characters from the Felder complex~\cite{Felder}.
The definition with screenings can be generalized to any root system. The algebra $\mathcal{W}$
coincides with the intersection of screenings kernels in the vacuum module of~$\mathcal{L}$.

The construction in terms of the bundle $\xi$ 
allows us to calculate characters of irreducible modules using the Lefschetz
formula. The equivariant Euler characteristics of $\xi$
is 
\begin{equation}\label{eq:lefshets}
 \sum_{i\geq0}(-1)^i\Tr_{H^i}(e^h q^{L_0})=\sum_{x\in S}
\Tr_{\xi_x\otimes\mathcal{F}_x}(e^h q^{L_0}),
\end{equation}
where $S$ is the set of fixed points of the standard torus action on the flag 
manifold~$\theG/\theB$,
$h$ is an element of the Cartan subalgebra, $L_0$ is the element of the Virasoro
subalgebra in~$\voa$, $H^i$ is i-th cohomology of $\xi$, $\xi_x$
is the fiber of $\xi$ at the point $x$ and $\mathcal{F}_x$ is the ring of formal series in 
a neighborhood of~$x$. 
Actually, all cohomology excepting $H^0$ are equal to zero and therefore the left hand side
of~(\ref{eq:lefshets}) is equal to the character of the irreducible $\voaw$-module.
Thus, we obtain expressions for the characters in terms of theta functions from~\cite{Kac-inf}
\begin{equation}\label{chi-lefshetz}
 \chi_\lambda(q)=\frac{1}{\eta(q)^\ell}\sum_{\omega\in\Gamma^{\vee+}}
  \frac{\zeta_\omega}{\prod_{i=1}^\ell\left(\scalar{\alpha_i}{\omega}!\right)}
   \Theta^\omega_\lambda(q)
\end{equation}
where
\begin{equation}
 \Theta^{\omega}_\lambda(q)=\left.\dd^\omega\Theta_\lambda(q,\z)\right|_{\z=1}
\end{equation}
is partial derivative of the theta function and
\begin{equation}
 \zeta_\omega=
 \sum_{\mu\in\Gamma^{\vee}}(-1)^{\scalar{\mu}{\rho^\vee}}
 \prod_{j=1}^\ell\left({\scalar{\alpha_j}{\omega}\atop\scalar{\alpha_j}{\mu}}\right)
  \prod_{\alpha\in\Delta^+}\left(1+\frac{\scalar{\alpha}{\mu}}{\scalar{\alpha}{\rho}}\right),
   \qquad\omega\in\Gamma^{\vee+}.
\end{equation}
See explanations of the notation in Sec.~\ref{sec:char}.
In the case ${\theG}=SL(2)$, (\ref{chi-lefshetz}) coincides with formulas for characters of irreducible 
$W$-modules in~\cite{FHST,Flohr2}.

The paper is organized as follows. In section~\ref{sec:prelim}, we introduce notations and 
recall known facts about vertex operator algebras. In section~\ref{sec:screenings}, we 
introduce screening operators. In section~\ref{sec:thevoa}, we introduce the main object of 
our investigation the vertex operator algebra $\voaw$ and describe its irreducible modules
as cohomologies of some bundles on homogeneous space. In section~\ref{sec:QG}, we introduce
a quantum group that presumably centralizes $\voaw$. In section~\ref{sec:char}, we calculate
characters of $\voaw$ irreducible modules.

\section{Preliminaries\label{sec:prelim}}
Let $\theg$ be a simply-laced semisimple Lie algebra of rank $\ell$, $\theh$ and $\theb$ be its Cartan
and Borel subalgebras respectively. Let $\theG$, $\theH$ and $\theB$ be Lie groups corresponding 
to $\theg$, $\theh$ and $\theb$ respectively.
Let $\Gamma$ be the root lattice, $\sroots$ be the set of simple positive roots
$\alpha_i$, $i=1,2,\dots,\ell$ and $\roots$ be the set of roots.
Let $\scalar{\cdot}{\cdot}$ be standard scalar 
product in $\theh^*$, 
$c_{ij}=\frac{2\scalar{\alpha_i}{\alpha_j}}{\scalar{\alpha_j}{\alpha_j}}$
be the Cartan matrix, which in the simply-laced case coincides with
the Gramm matrix and $c^{ij}$ be the inverse matrix to~$c_{ij}$.
Let $\rho$, $\scalar{\rho}{\alpha_i}=1$, $i,1,\dots,\ell$ be the half of the sum of positive roots.
Let $\omega_i$, $i=1,2,\dots,\ell$ be fundamental weights $\scalar{\omega_i}{\alpha_j}=\delta_{ij}$.
Let $\Gamma^\vee$ be the weight lattice. 
We set $\Omega=\Gamma^\vee/\Gamma$. We set $\gamma=|\Omega|$ the order of the group $\Omega$.

We choose representatives of the elements from the abelian group $\Omega$ in~$\Gamma^\vee$
in the following way.
For algebras $A_\ell$ we choose $0$, $\omega_i$ with $i=1,\dots,\ell$; for $D_\ell$ 
we choose $0$, $\omega_1$, $\omega_{\ell-1}$ and $\omega_\ell$; 
for $E_6$ we choose $0$, $\omega_1$ and $\omega_3$; for $E_7$ we choose $0$, $\omega_2$
 and for $E_8$ the group $\Omega$ contains only~$0$.

We consider the free scalar fields $\varphi_\alpha(z)$ for $\alpha\in\Gamma$ with the OPE
\begin{equation}\label{the-phi-ope}
 \varphi_\alpha(z)\varphi_\beta(w)=\scalar{\alpha}{\beta}\log(z-w).
\end{equation}
We note that between $\varphi_\alpha(z)$ there are $\ell$ linearly independent.
We also use notation $\varphi_i(z)=\varphi_{\alpha_i}(z)$, $i=1,\dots,\ell$.
We assume the mode decomposition
\begin{equation}
 \varphi_\alpha(z)=(\bar\varphi_\alpha)_0+(\varphi_\alpha)_0\log z
  -\sum_{n\neq0}\frac{1}{n}(\varphi_\alpha)_n z^{-n}.
\end{equation}
In order to have correct commutation relations between screening operators we introduce
nontrivial bracket of the constant modes of $\varphi_\alpha$
\begin{equation}\label{c-mode-comm}
 [(\bar\varphi_i)_0,(\bar\varphi_j)_0]=b_{ij},
\end{equation}
where
\begin{equation}
 b_{ij}=-b_{ji}=\begin{cases}
         1,\quad\mbox{$i<j$ and $i$-th and $j$-th nodes are connected in the Dynkin diagram},\\
         0,\quad\mbox{otherwise.}
        \end{cases}
\end{equation}
We define a bilinear form $\{\alpha,\beta\}$ on $\theh^*$ in the basis of simple roots
\begin{equation}\label{new-scalar}
 \{\alpha_i,\alpha_j\}=c_{ij}+b_{ij}.
\end{equation}
We note that this bilinear form is not symmetric nor antisymmetric.
We fix an integer $p\geq2$ and introduce the set of vertex operators
\begin{equation}\label{vertex}
 V_\lambda(z)=e^{\frac{1}{\sqrt{p}}\varphi_\lambda(z)},\quad \lambda\in\Gamma^\vee.
\end{equation}
After changing the commutation relations of the constant modes~(\ref{c-mode-comm}), we have
the following braiding of vertex operators
\begin{equation}\label{new-braid}
 V_\lambda(z)V_\mu(w)\sim \q^{\{\lambda,\mu\}}V_\mu(w)V_\lambda(z),
\end{equation}
where
\begin{equation}\label{theq}
 \q=e^{\frac{\pi i}{p}}.
\end{equation}
\begin{Rem}
Usually vertex operators are defined by (\ref{vertex}) with $[(\bar\varphi_i)_0,(\bar\varphi_j)_0]=0$,
which gives the braiding in the standard form~$\q^{\scalar{\lambda}{\mu}}$. After modification 
the braiding is changed.
We note that the monodromy of the vertex operators doesn't change and is equal 
to~$\q^{2\scalar{\lambda}{\mu}}$. See a similar construction in~\cite{Frenkel-Kac}.
\end{Rem}

We consider the lattice vertex operator algebra $\voa_\ell(p)$ 
corresponding to $\bar\Gamma=\sqrt{p}\,\Gamma$, 
where $\Gamma$ is the root lattice of the semisimple algebra of the type $X$ equals to $A$, $D$
or~$E$. In what follows to simplify notation we often write $\voa$.
The basis in $\voa$ consists of elements $P(\partial\varphi_\beta)V_{p\alpha}(z)$,
where $\alpha\in\Gamma$ and $P$ is a differential polynomial.
We choose the energy-momentum tensor of $\voa$ in the form
 \begin{equation}\label{theemt}
  T(z)=\half c^{ij}\partial\varphi_i(z)\partial\varphi_j(z)
  +Q_0\partial^2\varphi_\rho(z),
 \end{equation}
where
\begin{equation}
 Q_0=\sqrt{p}-\frac{1}{\sqrt{p}}.
\end{equation}
We choose the nonstandard background charge in order to have the energy-momentum tensor
commuting with screening operators, which are introduced in the next section.
The central charge is
\begin{equation}
 c=\ell+12\scalar{\rho}{\rho}(2-p-\frac{1}{p})=\ell+h\dim\!\theg\,(2-p-\frac{1}{p}),
\end{equation}
where $h$ is the Coxeter number of $\theg$.

The conformal dimension $\Delta_\lambda$ of the vertex operator $e^{\varphi_\lambda}$ with $\lambda\in\bar\Gamma^\vee$ 
is given by the expression
\begin{equation}\label{energy}
 \Delta_\lambda=\half\scalar{\lambda-Q_0\rho}{\lambda-Q_0\rho}+\frac{c-\ell}{24}.
\end{equation}

The $\voa_\ell(p)$ irreducible modules are enumerated by elements of the abelian 
group $\Lambda=\bar\Gamma^\vee/\bar\Gamma$.
We choose the basis $\lambda_j=\frac{1}{\sqrt{p}}\omega_j$, $j=1,2,\dots,\ell$ in $\bar\Gamma^\vee$.
 For each equivalence class $<\lambda>\in\Lambda$, a unique representative
  $\lambda\in\bar\Gamma^\vee$ of the form
\begin{equation}\label{canon-rep}
 \lambda=\sqrt{p}\omega+\sum_{j=1}^\ell(1-s_j)\lambda_j,
\end{equation}
where $\omega\in\Omega$ and $s_j=1,2,\dots,p$
can be chosen (See description of the representatives in the second paragraph of this section.). 
 We call the form (\ref{canon-rep}) of representatives canonical and use notation
\begin{align}
 \hat\lambda&=\omega,\label{hlamb}\\
 \bar\lambda&=\sum_{j=1}^\ell(1-s_j)\lambda_j.\label{blamb}
\end{align}
We note that $\Lambda$ can be described as an Abelian group with generators $\mu^i$ and
relations $\sum_{j=1}^\ell pc_{ij}\mu^j=0$. For $A_2$ root system and $p=3$ the group $\Lambda$ 
generators are shown in the diagram from Appendix~\ref{app:lambda}.

Let $\modF_\alpha$, $\alpha\in\Gamma^\vee$ be the Fock module corresponding to the 
vertex $V_\alpha(z)$. We set
\begin{equation}
 \modV_{<\lambda>}=\bigoplus_{\alpha\in\Gamma}\modF_{\lambda+\alpha}.
\end{equation}
Then, $\modV_{<\lambda>}$ for $<\lambda>\in\Lambda$ is an irreducible module of~$\voa_\ell(p)$~\cite{Kac-vertex,Frenkel-BenZui}.

\section{Screening operators\label{sec:screenings}}
We consider the screening operators
\begin{equation}\label{e-opers}
 e_i=\frac{1}{2\pi i}\oint dz e^{\sqrt{p}\varphi_i(z)}
\end{equation}
and
\begin{equation}\label{F-opers}
 F_i=\frac{1}{2\pi i}\oint dz e^{-\frac{1}{\sqrt{p}}\varphi_i(z)}
\end{equation}
which  commutes with the Virasoro algebra~(\ref{theemt}). 
We note that because (\ref{new-braid}) operators $e_i$ commute with~$F_j$ for~$i,j=1,\dots,\ell$.

We note that $e_i$ acts in each $\modV_\lambda$. To define the space
in which operators $F_i$ act we should  introduce dressed vertex operators~\cite{Felder}.
In this paper we need only that $F_i$ is a well defined operator 
from $\modV_0$ to $\modV_{-\lambda_i}$.

Actually, the subalgebra of $\voa_\ell(p)$ consisting of zero momentum fields
commuting with screenings (\ref{e-opers}) is the $W$-algebra $\voahr_\ell(p)$
obtained by Hamiltonian reduction from affine algebra $\thega_k$ with
\begin{equation}
 p=k+h.
\end{equation}
A description of the $\voahr_\ell(p)$ representation category can be found in \cite{Arakawa}.

\section{The vertex operator algebra $\voaw_\ell(p)$\label{sec:thevoa}}
In the section, we define the main object of the paper the vertex operator algebra $\voaw_\ell(p)$,
where $X$ means $A$, $D$ or $E$ type of simply-laced semisimple Lie algebra.
Thus we have $\voaw[A]_\ell(p)$ for $\ell=1,2,3,\dots$, $\voaw[D]_\ell(p)$ for $\ell=4,5,6,\dots$
and $\voaw[E]_\ell(p)$ for $\ell=6,7,8$. The vacuum representation of the 
algebra $\voaw_\ell(p)$ can be defined as an 
intersection of kernels of operators $F_i$ in~$\modV_0$.
 We give another definition of $\voaw_\ell(p)$ in terms of operators~$e_i$.

\subsection{The action of $\theb$ in the irreducible $\voa_\ell(p)$-modules.}
\begin{Thm}\label{b-action}
\begin{enumerate}
 \item \label{e-h} 
 The $\voa_\ell(p)$ module $\modV_\lambda$ for $\lambda\in\Lambda$ admits the action of $\theb$
 given by the standard generators
  \begin{align}\label{e-on-V}
   e_i&=\frac{1}{2\pi i}\oint dz e^{\sqrt{p}\varphi_i(z)},\\
   h_i&=\frac{1}{2i\pi\sqrt{p}}\oint dz\partial\varphi_i-\frac{1}{\sqrt{p}}\scalar{\alpha_i}{\bar\lambda}
   + \scalar{\alpha_i}{\mu}\label{h-on-V}
  \end{align}
with $i=1,\dots,\ell$, $\mu\in\Gamma^\vee$ and $\bar\lambda$ defined in~(\ref{blamb}).
\item \label{B-integrate} The action of the Borel subalgebra $\theb$ given in part~(\ref{e-h})
of the Theorem is integrated to the action of $\theB$ in~$\modV_\lambda$.
\end{enumerate}
\end{Thm}
\begin{Rem}\label{rem:mu}
We note that~(\ref{e-on-V}) and (\ref{h-on-V}) define an action of $\theb$ on 
$\voa_\ell(p)$ by infinitesimal symmetries. Therefore, we can construct a semidirect product
$U(\theb)\ltimes\voa_\ell(p)$ of the universal enveloping $U(\theb)$ of the Borel subalgebra
and the VOA $\voa_\ell(p)$.
Moreover, we can define on $\voa_\ell(p)$-modules 
$\modV_\lambda$ a $U(\theb)\ltimes\voa_\ell(p)$-module structure.
 We let $\modV_\lambda(\mu)$ denote the $U(\theb)\ltimes\voa_\ell(p)$-module
defined by~(\ref{e-on-V}) and (\ref{h-on-V}).
We introduce 1-dimensional $\theb$ module $\oneb(\mu)$, $\mu\in\theh^*$ on which $e_i$ acts by zero
and $h_i$ by multiplication with $\scalar{\alpha_i}{\mu}$. We note 
$\modV_\lambda(\mu)=\modV_\lambda(0)\otimes\oneb(\mu)$. We also use notation
$\modV_\lambda$ for $\modV_\lambda(0)$.
\end{Rem}
\begin{proof}[Proof of the Theorem~\ref{b-action}]
\begin{enumerate}
 \item   The relations $\commut{h_i}{e_j}=c_{ij}e_j$ are checked by simple calculation 
using (\ref{the-phi-ope}).

To check the Serre relations 
\begin{equation}
 \ad_{e_i}^{1-c_{ij}}e_j=0,\qquad i,j=1,\dots,\ell,\quad i\neq j
\end{equation}
we consider $\ad_{e_i}^{1-c_{ij}}e^{\sqrt{p}\varphi_i(z)}$.
We note that $\ad_{e_i}^{1-c_{ij}}e^{\sqrt{p}\varphi_i(z)}=
 P(\partial\varphi_k)e^{\sqrt{p}((1-c_{ij})\varphi_i(z)+\varphi_j(z))}$, where
$P$ is a differential polynomial in $\partial\varphi_k$, $k=1,\dots,\ell$.
This statement is true only when we chose the braiding~(\ref{new-braid}).
The polynomial $P$ should have a nonnegative conformal weight $\Delta_P$. The field
$\ad_{e_i}^{1-c_{ij}}e^{\sqrt{p}\varphi_i(z)}$ has conformal weight $1$ because
$e_i$ are screening operators and $e^{\sqrt{p}\varphi_i(z)}$ is a screening current.
A direct calculation with (\ref{theemt}) gives the conformal weight of
$e^{\sqrt{p}((1-c_{ij})\varphi_i(z)+\varphi_j(z))}$ being equal to $2-c_{ij}$,
which means that the balance of weights $1=2-c_{ij}+\Delta_P$ can not be satisfied
with a nonnegative weight~$\Delta_P$. Thus, $\ad_{e_i}^{1-c_{ij}}e^{\sqrt{p}\varphi_i(z)}=0$.

\item  The statement follows from the observation that $\voa_\ell(p)$ decomposes into
a direct sum of integrable finite dimensional representations of $\theb$.
\end{enumerate}
\end{proof}

\subsection{Bundles on the homogeneous space $\theG/\theB$}
We consider the bundle 
\begin{equation}
 \xi_\lambda(\mu)=\theG\times_{\theB}\modV_\lambda(\mu),\qquad\lambda\in\Lambda,\quad\mu\in\Gamma^\vee
\end{equation}
on the homogeneous space $\theF=\theG/\theB$, where the action of $\theB$ on $\theG$ is given 
by the right multiplication and on $\modV_\lambda$ by Theorem~\ref{b-action} with the 
corresponding~$\mu$. We set $\xi_\lambda=\xi_\lambda(0)$.
We let $\mathcal{O}(\mu)$ denote the standard 1-dimensional bundle on $\theF$ 
\begin{equation}
 \mathcal{O}(\mu)=\theG\times_{\theB}\oneb(\mu)
\end{equation}
with $\oneb(\mu)$ defined in Remark~\ref{rem:mu}.
\begin{Prop}\label{prop:xi-mu}
  \begin{equation}
 \xi_\lambda(\mu)=\xi_\lambda\otimes\mathcal{O}(\mu)
\end{equation}
\end{Prop}
\begin{proof}
 An immediate consequence of (\ref{h-on-V}).
\end{proof}

\begin{Thm}\label{thm:cohomol}
 \begin{enumerate}
 \item $H^n(\xi_\lambda)=0$, for $n>0$\label{item1}
 \item $H^0(\xi_\lambda)$ is embedded into the fiber $\modV_\lambda$
of the bundle $\xi_\lambda$ over any point.
\end{enumerate}
\end{Thm}
The proof of the Theorem is based on a calculation of the cohomologies in the $A_1$ case.
To do that we should recall notations from~\cite{FGST2}.
In~\cite{FGST2}, indecomposable modules of the quantum group $\overline{\mathscr{U}}_{\q} s\ell(2)$
were described. We need the following of them 
\begin{itemize}
 \item for $a=\pm$ and $s=1,\dots,p$, $\mathscr{X}^a_s$ are $s$ dimensional
 irreducible modules;
\item for $a=\pm$, $s=1,\dots,p-1$ and integer $n\geq2$, the module $\mathscr{M}^a_s(n)$
is an indecomposable module with socle $\oplus_1^n\mathscr{X}^a_s$ and the 
quotient~$\oplus_1^{n-1}\mathscr{X}^{-a}_{p-s}$;
\item for $a=\pm$, $s=1,\dots,p-1$ and integer $n\geq2$, the module $\mathscr{W}^a_s(n)$
is an indecomposable module with socle $\oplus_1^{n-1}\mathscr{X}^{-a}_{p-s}$ and the 
quotient~$\oplus_1^{n}\mathscr{X}^{a}_{s}$.
\end{itemize}
Taking an equivalence of $\overline{\mathscr{U}}_{e^{i\pi/p}} s\ell(2)$ and $\voaw[A]_1(p)$
representation categories~\cite{NTsu} into account we assume the same notations 
for the corresponding $\voaw[A]_1(p)$ modules. In what follows we need also
condensed notation.
For $\lambda=\frac{1-a}{2}\sqrt{p}\omega_1+(1-s)\lambda_1$, we set 
\begin{gather}
 \mathfrak{M}_\lambda(\mu)=\mathscr{X}^a_s, \qquad \mbox{\rm for $\mu=0$},\\ 
 \mathfrak{M}_\lambda(\mu)=\mathscr{M}^{-a}_{p-s}(\mu+1), \qquad \mbox{\rm for $\mu>0$},\\
 \mathfrak{M}_\lambda(\mu)=\mathscr{X}^{a}_{s}, \qquad \mbox{\rm for $\mu=-1$},\\
 \mathfrak{M}_\lambda(\mu)=\mathscr{W}^a_s(-\mu), \qquad \mbox{\rm for $\mu<-1$}
\end{gather}
We also set $\modX_\lambda=\mathscr{X}^a_s$.
\begin{Lemma}\label{sl2-case}
 In $A_1$ case we have
\begin{itemize}
 \item[{\rm for $\mu\geq0$}] 
  \begin{gather}\label{H0mu}
   H^0(\xi_\lambda(\mu))=\mathfrak{M}_\lambda(\mu),\\
   H^1(\xi_\lambda(\mu))=0,
  \end{gather}
\item[{}\ {\rm for $\mu\leq-1$}]
\begin{gather}
   H^0(\xi_\lambda(\mu))=0,\\
   H^1(\xi_\lambda(\mu))=\mathfrak{M}_\lambda(\mu),\label{H1mu}
  \end{gather}
\end{itemize}
\end{Lemma}
\begin{proof}
 The irreducible $\voa[A]_1(p)$-module $\modV_\lambda(0)$ is a reducible module 
of~$\voaw[A]_1(p)$~\cite{FGST1} and its structure can be described
by the following exact sequence
\begin{equation}\label{eq:ex-seq}
 0\to\modX_\lambda\to\modV_\lambda(0)\to\modX_{\alpha_1/\sqrt{p}-\lambda}\to0.
\end{equation}
The spaces $\modX_\lambda$ and $\modX_{\alpha_1/\sqrt{p}-\lambda}$
are irreducible $\voaw[A]_1(p)$-modules and at the same time they  
bear the action of $\theb$ induced by the action of $\theb$ on $\modV_\lambda(0)$.
The action of~$\theb$ on $\modX_\lambda$ is extended to an action
of $s\ell(2)$~\cite{Kausch}. 
This means that the bundle $\xi_\lambda(0)$ contains the trivial subbundle
 $\bar\xi_\lambda(0)=\oCP^1\times\modX_\lambda$. 
We note that the action of~$\theb$ on the quotient $\modX_{\alpha_1/\sqrt{p}-\lambda}$
is not extended to an $s\ell(2)$ action but the action of~$\theb$ on 
$\modX_{\alpha_1/\sqrt{p}-\lambda}\otimes\oneb(1)$ is extended to the $s\ell(2)$ action.
For the quotient bundle we have
\begin{equation}
\xi_\lambda(0)/\bar\xi_\lambda(0)=\left(\oCP^1\times\modX_{\alpha_1/\sqrt{p}-\lambda}\right)
\otimes\mathcal{O}(-1).
\end{equation}
(The RHS means tensor product of the trivial bundle $\oCP^1\times\modX_{\alpha_1/\sqrt{p}-\lambda}$
with~$\mathcal{O}(-1)$.)
This bundle has zero cohomology because $H^i(\mathcal{O}(-1))=0$.
This shows the Lemma for $\mu=0$ and $\mu=-1$.
Other statements of the Lemma are obtained by multiplying with~$\mathcal{O}(\mu)$.
Obviously, we have $H^0(\xi_\lambda(\mu))=H^0(\xi_\lambda(0))\otimes\oC^{\mu}$
for~$\mu>0$ and $H^1(\xi_\lambda(\mu))=H^1(\xi_\lambda(-1))\otimes\oC^{-\mu}$
for~$\mu<0$. Indecomposability of modules appearing in the RHS of~(\ref{H0mu})
and~(\ref{H1mu}) follows from the observation that the cohomology has nontrivial
mappings on the corresponding Verma modules.
\end{proof}

\begin{proof}[Proof of the Theorem~\ref{thm:cohomol}]
The proof of the Theorem resembles a proof of Bott-Borel-Weyl theorem and 
we give only a brief description of it. 
Let $\sigma_\alpha$ be the Weyl group element corresponding to the simple root~$\alpha$.
Let $\sigma_{\alpha_1}\sigma_{\alpha_2}\dots\sigma_{\alpha_n}$ be the reduced decomposition
of the longest element in the Weyl group into the product of simple reflections.
We note that~$n$ is the dimension of the flag manifold~$\theF$. 
We define the shifted action of the Weyl
group~$\sigma_\alpha\cdot\omega=\sigma_\alpha(\omega+\sqrt{p}\rho)-\sqrt{p}\rho$.
Let $\theP_\alpha$ be the parabolic subalgebra
corresponding to the simple root $\alpha$ and $\theF_\alpha=\theG/\theP_\alpha$.
We note that $\theF$ is fibered over $\theF_\alpha$
with fibers $\oCP^1$. We let $\pi$ denote the projection $\theF\to\theF_\alpha$
and $\xi_\lambda^\alpha=\pi_*\xi_\lambda$ the direct image of~$\xi_\lambda$.

We consider two bundles $\xi_\lambda$ and $\xi_{\sigma_\alpha\cdot\lambda}$.
Their direct images $\xi_\lambda^\alpha$ and $\xi_{\sigma_\alpha\cdot\lambda}^\alpha$
are the same bundle on~$\theF_\alpha$. In more details,
taking Lemma~\ref{sl2-case} into account, for $\lambda\in\Lambda$, we obtain
$H^0(\left.\xi_\lambda\right|_{\oCP^1})=
H^1(\left.\xi_{\sigma_\alpha\cdot\lambda}\right|_{\oCP^1})\neq0$
and $H^1(\left.\xi_\lambda\right|_{\oCP^1})=
H^0(\left.\xi_{\sigma_\alpha\cdot\lambda}\right|_{\oCP^1})=0$.
This gives $H^i(\xi_\lambda)=H^{i+1}(\xi_{\sigma_\alpha\cdot\lambda})$
by the standard Leray spectral sequence.
We repeat this procedure for the longest element in the Weyl group
and obtain 
$H^i(\xi_\lambda)=H^{i+k}(\xi_{\sigma_{\alpha_1}\cdot\sigma_{\alpha_2}\cdot
 \dots\cdot\sigma_{\alpha_k}\cdot\lambda})$.
But this requires $H^i(\xi_\lambda)=0$ for~$i>0$.
\end{proof}
Taking Theorem~\ref{thm:cohomol} into account, we assume the following definition.
\begin{Dfn}
 The vacuum module $\modX_0$ of the vertex-operator algebra $\voaw_\ell(p)$ is $H^0(\xi_0)$.
\end{Dfn}
Other irreducible $\voaw_\ell(p)$-modules can be obtained as follows 
\begin{equation}\label{X-cohomol}
 \modX_\lambda=H^0(\xi_\lambda),\qquad\lambda\in\Lambda.
\end{equation}

\subsection{Generators of the vertex operator algebra $\voaw_\ell(p)$\label{sec:generators}}
The system of generators $\voaw_\ell(p)$ consists of two subsets.
The first subset of generators are generators of the W-algebra $\voahr_\ell(p)$ 
(see Sec.~\ref{sec:screenings}),
which is a subalgebra in~$\voaw_\ell(p)$. The algebra $\theg$ acts trivially on~$\voahr_\ell(p)$.
The second subset of generators span the adjoint representation of~$\theg$.

We describe the second subset of $\voaw_\ell(p)$ generators in details.
For the generators, we introduce notation $W^\alpha$ with $\alpha\in\roots$ 
 and $W^{0,\alpha}$ with $\alpha\in\sroots$.
Let $t^{\alpha}$ for $\alpha\in\roots$ be the basis vector from weight subspace with the weight $\alpha$ 
 and $t^{0,\alpha}$ with $\alpha\in\sroots$ be the basis in the zero weight subspace
 of the adjoint representation of $\theg$.
Let $\theta\in\roots$ be the lowest root of $\theg$, i.e.{} $\theta-\alpha\not\in\roots$ for 
any $\alpha\in\roots^+$. 
Then $t^{\theta}$ be the lowest weight vector in the adjoint representation of $\theg$.
Any root $\alpha\in\roots$ can be written in the
form
\begin{equation}
 \alpha=\theta+\sum_{i=1}^\ell n_i\alpha_i,\qquad n_i\in\oN_0.
\end{equation}
Then,
\begin{equation}\label{non-zero-weights}
 t^\alpha=\ad_{e_{j_1}}\ad_{e_{j_2}}\dots\ad_{e_{j_m}}t^{\theta}
\end{equation}
where $e_i$ appears presicely $n_i$ times and 
\begin{equation}\label{zero-weights}
 t^{0,\alpha}=\ad_{e_{j_1}}\ad_{e_{j_2}}\dots\ad_{e_{j_m}}t^{\theta}
\end{equation}
where $e_i$ appears presicely $a_i$ times with $a_i$ be labels in the Dynkin diagram.
Then, we set
\begin{equation}
 W^\alpha(z)=\ad_{e_{j_1}}\ad_{e_{j_2}}\dots\ad_{e_{j_m}}e^{\sqrt{p}\varphi_\theta(z)}
\end{equation}
with presicely the same product of adjoint operators as in (\ref{non-zero-weights})
and 
\begin{equation}
 W^{0,\alpha}(z)=\ad_{e_{j_1}}\ad_{e_{j_2}}\dots\ad_{e_{j_m}}e^{\sqrt{p}\varphi_\theta(z)}
\end{equation}
with presicely the same product of adjoint operators as in (\ref{zero-weights}).

The fields $W^\alpha(z)$ and $W^{0,\alpha}(z)$ have the same conformal dimensions equal to
\begin{equation}
 \frac{p}{2}c^{ij}\scalar{\alpha_i}{\theta}\scalar{\alpha_j}{\theta}+(1-p)\scalar{\rho}{\theta}=h(p-1)+1.
\end{equation}

\section{Quantum group\label{sec:QG}}
In this section we describe the quantum group $\bar{U}_\q(X_\ell)$.
We conjecture that its representation category is equivalent 
to the representation category of the vertex operator algebra $\voaw_\ell(p)$.

Suppose we have an algebra graded by $\theh^*$. We introduce q-bracket or q-adjoint action.
Let $x$ and $y$ are two homogeneous elements from the algebra, then
\begin{equation}
 \qad_xy=xy-\q^{\{\#x,\#y\}}yx,
\end{equation}
where $\q$ is given by (\ref{theq}), the scalar product $\{\cdot,\cdot\}$ is defined 
in~(\ref{new-scalar}) and $\#x\in\theh^*$ is the weight of $x$.
The most important example is the case, where $x$ and $y$ belong to the algebra 
generated by vertex operators and grading operators are commutators with zero modes
\begin{equation}
 h_i=\sqrt{p}(\varphi_{\omega_i})_0,\qquad i=1,\dots,\ell,
\end{equation}
where $\omega_i$ are fundamental weights.

The quantum group $\bar{U}_\q(X_\ell)$ is an associative algebra with 
generators $E_i$, $F_i$, $K_i$, $K_i^{-1}$ for $i=1,\dots,\ell$ and $1$.
This algebra is graded by the root lattice $\Gamma$ and weights of generators
are 
\begin{equation}
 \#E_i=\alpha_i,\quad\#F_i=-\alpha_i,\quad\#K_i=0.
\end{equation}
To describe the relations, we introduce elements
\begin{equation}
 L_i=\prod_{j=1}^\ell K_j^{c_{ij}}.
\end{equation}
The relations are
\begin{gather}
 E_i^p=F_i^p=0,\quad
 K_iK_j=K_jK_i,\quad L_i^{p}=1,\\
 K_iE_jK_i^{-1}=q^{2\delta_{ij}}E_j,\quad
 K_iF_jK_i^{-1}=q^{-2\delta_{ij}}F_j,\\
 \qad_{E_i}{F_j}=\delta_{ij}\frac{L_i-1}{\q-\q^{-1}},\\
 \qad_{E_i}^{1-c_{ij}}{E_j}=0,\quad\qad_{F_i}^{1-c_{ij}}{F_j}=0,\qquad i\neq j.
\end{gather}
This algebra is the centralizer of $\voaw_\ell(p)$ and we conjecture that the two algebras
have equivalent representation categories. We also note that for $\ell=1$ This algebra is
isomorphic to the quantum group from~\cite{FGST1}.

The algebra $\bar{U}_\q(X_\ell)$ is a braided Hopf algebra. To describe the braided Hopf 
algebra structure, we introduce the braided tensor 
product $\bar\otimes$.
The algebra $\bar{U}_\q(X_\ell)\bar\otimes\bar{U}_\q(X_\ell)$ differs from
$\bar{U}_\q(X_\ell)\otimes\bar{U}_\q(X_\ell)$ only by commutation relations between two
multipliers. In standard tensor product elements of the form $x\otimes1$ and $1\otimes y$
commute while elements $x\bar\otimes1$ and $1\bar\otimes y$ q-commute
\begin{equation}
 \qad_{x\bar\otimes1}1\bar\otimes y=0,
\end{equation}
where $\#(x\bar\otimes y)=\#x+\#y$. The comultiplication 
$\bar\Delta:\bar{U}_\q(X_\ell)\to\bar{U}_\q(X_\ell)\bar\otimes\bar{U}_\q(X_\ell)$ is given by
\begin{equation}
 \bar\Delta(K_i)=K_i\bar\otimes K_i
\end{equation}
and
\begin{equation}
 \bar\Delta(x)=x\bar\otimes1+1\bar\otimes x
\end{equation}
for $x=E_i$ or $F_i$.

At the end of the section, we note that the appearance of the braided Hopf algebra structure
insted of Hopf algebra structure probably solves a contradiction observed in~\cite{[KoSai]}.
In~\cite{[KoSai]} it was observed that the tensor product of 
$\overline{\mathscr{U}}_{\q} s\ell(2)$ modules is not necessarily commutative although
representation category of $\overline{\mathscr{U}}_{\q} s\ell(2)$ is equivalent~\cite{NTsu}
to the representation category of $(1,p)$ model vertex operator algebra which is braided tensor
category by construction.

\section{Characters\label{sec:char}}
In what follows we use notation
\begin{equation}
 \z^\alpha=\prod_{j=1}^\ell z_j^{\scalar{\alpha_j}{\alpha}},\qquad\alpha\in\theh^*.
\end{equation}
We also use notation
\begin{equation}
 \z[\alpha]=(z_1^{\scalar{\alpha_1}{\alpha}},z_2^{\scalar{\alpha_2}{\alpha}},\dots,
  z_\ell^{\scalar{\alpha_\ell}{\alpha}}),
\end{equation}
where in the RHS we have a list of monomials.
In particular, $\z[\rho]=(z_1,z_2,\dots,z_\ell)$.
For example, for a function $f$, $f(\z[\alpha])$ means
$f(z_1^{\scalar{\alpha_1}{\alpha}},z_2^{\scalar{\alpha_2}{\alpha}},\dots,
  z_\ell^{\scalar{\alpha_\ell}{\alpha}})$.
We also use notation 
\begin{equation}
\dd^\alpha=\prod_{j=1}^\ell \left(\frac{\partial}{\partial z_j}\right)^{\scalar{\alpha_j}{\alpha}}
\end{equation}
for the derivatives with respect to $z_i$.

We define the shifted action of the Weyl group $W$ on $\z$ by the formulas
\begin{equation}
 w(\z^\alpha)=\z^{w^{-1}\cdot\alpha},\quad w(f(\z[\alpha]))=f(\z[w^{-1}\cdot\alpha]),
\end{equation}
where $w\cdot\alpha=w(\alpha-\rho)+\rho$ and $f$ is a function in $\z$.

We introduce the denominator
\begin{equation}
 d(\z[\rho])=\prod_{\alpha\in\Delta^-}(1-\z^{\alpha})
\end{equation}
It satisfies
\begin{equation}
 w(d(\z[{\rho}]))=\epsilon(w)d(\z[\rho]),
\end{equation}
where $\epsilon(w)=(-1)^{\#w}$, where $\#w$ is the parity of $w$.
The Weyl formula for the character of the irreducible $\theg$-module with
the highest weight $\lambda$ is
\begin{equation}
 \chi^\theg_\lambda(\z^\rho)=\frac{1}{d(\z^{\rho})}\sum_{w\in W}\epsilon(w)w[\z^{\lambda}].
\end{equation}

Let V be a linear space equipped with an action of $L_0$ and $h_1,\dots,h_\ell\in\theh$.
Then the character of $V$ is
\begin{equation}
 \chi_{V}(q,\z^\rho)=\Tr_V\,q^{L_0-\frac{c}{24}}z_1^{h_1}z_2^{h_2}\dots z_\ell^{h_\ell}.
\end{equation}
In what follows we take $V$ be different VOA modules.

To write the characters we introduce the theta functions
\begin{equation}\label{theta}
 \Theta_\lambda(q,\z[\rho])=\sum_{\alpha\in\Gamma}
  q^{\frac{1}{2}\scalar{\sqrt{p}\alpha+\lambda-Q_0\rho}{\sqrt{p}\alpha+\lambda-Q_0\rho}}
  \z^{\alpha+\hat\lambda},\qquad\lambda\in\Lambda.
\end{equation}
We note that these theta functions differs from~\cite{Kac-inf} by the factor~$\z^{\bar\lambda}$.
(See definitions of $\bar\lambda$ and $\hat\lambda$ in~(\ref{hlamb}) and~(\ref{blamb}).)

The character of the irreducible $\voa_\ell(p)$-module $\modV_\lambda$ is
\begin{equation}
 \psi_\lambda(q,\z[\rho])=\frac{\Theta_\lambda(q,\z[\rho])}{\eta(q)^\ell},\qquad\lambda\in\Lambda.
\end{equation}

To calculate the character $\chi_\lambda(q,\z[\rho])$ of the 
irreducible $\voaw_\ell(p)$-module $\modX_\lambda$
we use (\ref{X-cohomol}) and Lefschetz formula.
We defined $\modX_\lambda$ in~(\ref{X-cohomol}) so the character of $\modX_\lambda$ is
\begin{equation}
\chi_\lambda(q,\z[\rho])=\Tr_{H^0(\xi_\lambda)}\,q^{L_0-\frac{c}{24}}
z_1^{h_1}z_2^{h_2}\dots z_\ell^{h_\ell}=\sum_{w\in W}\Tr_{\modM_w}\,q^{L_0-\frac{c}{24}}
z_1^{h_1}z_2^{h_2}\dots z_\ell^{h_\ell}
\end{equation}
This expression is a sum over fixed points of the standard action of the torus $\theH$ on
the flag manifold $\theF$. The fixed points are in one to one correspondence with the elements
of the Weyl group~$W$. Contribution of each fixed point $x_w$ is determined in the following way.
We consider the space of sections of the bundle $\xi_\lambda$ in the formal neighborhood of~$x_w$.
Contribution of the point is $\Tr_{\modM_w}\,q^{L_0-\frac{c}{24}}
z_1^{h_1}z_2^{h_2}\dots z_\ell^{h_\ell}$. The space $\modM_w$ as a module of 
$L_0$ and $h_1$, \dots, $h_\ell$ is the tensor product $\modV_\lambda\otimes\mathcal{F}_{x_w}$,
where $\mathcal{F}_{x_w}$ is the ring of formal series in~$x_w$.
Thus,
\begin{equation}
 \Tr_{\modM_w}\,q^{L_0-\frac{c}{24}}z_1^{h_1}z_2^{h_2}\dots z_\ell^{h_\ell}=
 \frac{\psi_\lambda(q,\z[{w(\rho)}])}{d(\z[{w(\rho)}])}.
\end{equation}
Summarizing, we have
\begin{equation}
 \chi_\lambda(q,\z[\rho])=\sum_{w\in\W}
  \frac{\psi_\lambda(q,\z[{w(\rho)}])}{d(\z[{w(\rho)}])}=\frac{1}{d(\z[{\rho}])}
  \sum_{w\in\W}\epsilon(w)\psi_\lambda(q,\z[{w(\rho)}]).
\end{equation}
We rewrite this expression for characters in different form.
\begin{Thm}\label{thm:w-decomp}
 The character $\chi_\lambda(q,\z^\rho)$ of the irreducible $\voaw_\ell(p)$-module $\modX_\lambda$
can be written in the form
\begin{equation}\label{eq:w-decomp}
 \chi_\lambda(q,\z^\rho)=\sum_{\alpha\in\Gamma^+}\chi^\theg_{\hat\lambda+\alpha}(\z[\rho])
  \chi^{\mathcal{W}}_{\lambda,\alpha}(q),
\end{equation}
where $\Gamma^+$ is the intersection of the root lattice with the positive Weyl chamber,
$\chi^\theg_{\hat\lambda+\alpha}(\z^\rho)$ is the character of the irreducible $\theg$-module
with the highest weight $\hat\lambda+\alpha$ ($\hat\lambda$ is defined in (\ref{hlamb}))
and $\chi^{\mathcal{W}}_{\lambda,\alpha}(q)$ is the character of the irreducible
$\voahr_\ell(p)$-module given by~\cite{Arakawa}
\begin{equation}\label{W-char}
 \chi^{\mathcal{W}}_{\lambda,\alpha}(q)=\frac{1}{\eta(q)^\ell}\sum_{w\in W}
  q^{\frac{1}{2}\scalar{\sqrt{p}w(\alpha)+\lambda
  +\frac{1}{\sqrt{p}}\rho}{\sqrt{p}w(\alpha)+\lambda+\frac{1}{\sqrt{p}}\rho}}.
\end{equation}
\end{Thm}
We note that (\ref{W-char}) are $\voahr_\ell(p)$-characters of general type, i.e.{}
irreducible modules with such characters exist for any value of the central charge.

The Theorem~\ref{thm:w-decomp} allows us to prove that $\voaw_\ell(p)$-module 
$\modX_\lambda$ is irreducible. Taking (\ref{eq:w-decomp}) into account,
we have 
\begin{equation}
 \modX_\lambda=\bigoplus_{\alpha\in\Gamma^+}\mathcal{R}_{\hat\lambda+\alpha}\otimes
  \mathfrak{Y}_{\lambda,\alpha},
\end{equation}
where $\mathcal{R}_{\hat\lambda+\alpha}$ is irreducible $\theg$ module and 
$\mathfrak{Y}_{\lambda,\alpha}$ is irreducible $\voahr_\ell(p)$ module.
Then the absence of invariant subspaces of $\voaw_\ell(p)$ action can be obtained
from a consideration of the action of generators from Sec.~\ref{sec:generators}.

We also give formulas for the characters of irreducible $\voaw_\ell(p)$ modules
adapted for calculating their modular properties and calculating values of characters
at $\z=1$. For this, we introduce derivatives of theta functions
\begin{equation}
 \Theta^{\omega}_\lambda(q)=\left.\dd^\omega\Theta_\lambda(q,\z)\right|_{\z=1},
   \qquad\omega\in\Gamma^{\vee+}.
\end{equation}
Then the characters can be written in the following nice form
\begin{equation}\label{char-in-theta}
 \chi_\lambda(q)=\frac{1}{\eta(q)^\ell}\sum_{\omega\in\Gamma^{\vee+}}
  \frac{\zeta_\omega}{\prod_{i=1}^\ell\left(\scalar{\alpha_i}{\omega}!\right)}
   \Theta^\omega_\lambda(q),
\end{equation}
where
\begin{equation}
 \zeta_\omega=
 \sum_{\mu\in\Gamma^{\vee}}(-1)^{\scalar{\mu}{\rho^\vee}}
 \prod_{j=1}^\ell\left({\scalar{\alpha_j}{\omega}\atop\scalar{\alpha_j}{\mu}}\right)
  \prod_{\alpha\in\Delta^+}\left(1+\frac{\scalar{\alpha}{\mu}}{\scalar{\alpha}{\rho}}\right),
   \qquad\omega\in\Gamma^{\vee+}
\end{equation}
and  $\left(n\atop m\right)$ are binomial coefficients.

\section{Conclusions}
A natural development of the results is a calculation of the $SL(2,\oZ)$ action
on the characters~(\ref{char-in-theta}). We belive that the $SL(2,\oZ)$ action 
will produce finite number of characters and pseudocharacters, whose
linear span will give a finite dimensional center $\bar{\mathcal{Z}}$ of the conformal
model. It is also interesting to calculate the center $\mathcal{Z}$ of the
quantum group $\bar{U}_\q(X_\ell)$ and the action of $SL(2,\oZ)$ on it.
We belive that $\bar{\mathcal{Z}}$ is isomorphic to $\mathcal{Z}$ as a representation
of~$SL(2,\oZ)$.

It is also interestingly to generalize the results of the paper to non simply laced case
and to the case of Lie superalgebras. Another interesting direction of investigations is
a generalization of our results to the case of the root lattice scaled not to $\sqrt{p}$
but to $\sqrt{\frac{p}{p'}}$. 
In the case of $A_1$, the algebra constructed in~\cite{FGST3} can be realized in the zero 
cohomology of a bundle on $\oCP^1$. Analogous construction can be done in the case of
arbitrary Lie algebra.
\subsubsection*{Acknowledgments}
We are grateful to A.A. Belavin for his kind attention to us.
The paper was supported by the RFBR-CNRS grant 09-02-93106. The work of BLF
was supported in part by the RFBR  initiative interdisciplinary project grant 09-02-12446-ofi\_m.
The work of IYuT was supported in part by the RFBR Grant No.08-02-01118.

\newpage

\appendix

\section{Group $\Lambda$\label{app:lambda}}
In the following picture long segments of strait lines are roots of $A_2$,
short segments of strait lines belong to the weight lattice of $A_2$, $\bullet$s
correspond to elements of the group $\Lambda$. Circled $\bullet$s correspond
to Steinberg modules (irreducible $\voaw_\ell(p)$ modules that coincide with
irreducible $\voa_\ell(p)$ modules).

\begin{picture}(400,400)(0,0)
 \put(200,200){\qbezier(0,0)(0,58)(0,116)}
 \put(200,200){\qbezier(0,0)(0,-58)(0,-116)}
\put(200,200){\qbezier(0,0)(-50.2,29)(-100.4,58)}
\put(200,200){\qbezier(0,0)(50.2,-29)(100.4,-58)}
 \put(200,200){\qbezier(0,0)(100,0)(200,0)}
\put(200,200){\qbezier(0,0)(-100,0)(-200,0)}
\put(200,200){\qbezier(0,0)(50,86)(100,174)}
\put(200,200){\qbezier(0,0)(-50,-86)(-100,-174)}
 \put(200,200){\qbezier(0,0)(50,30)(100,58)}
 \put(200,200){\qbezier(0,0)(-50,-30)(-100,-58)}
\put(200,200){\qbezier(0,0)(-50,86)(-100,172)}
\put(200,200){\qbezier(0,0)(50,-86)(100,-172)}
\put(200,200){\circle*{5}}
\put(200,161){\circle*{5}}
\put(200,122.6){\circle*{5}}
\put(166.5,180.0){\circle*{5}}
\put(133,161){\circle*{5}}
\put(166.5,141.3){\circle*{5}}
\put(166.5,102.6){\circle*{5}}
\put(133,122.3){\circle*{5}}
\put(133,83.7){\circle*{5}}
\put(133,83.7){\circle{7}}
\put(100,58){\put(200,200){\circle*{5}}
\put(200,161){\circle*{5}}
\put(200,122.6){\circle*{5}}
\put(166.5,180.0){\circle*{5}}
\put(133,161){\circle*{5}}
\put(166.5,141.3){\circle*{5}}
\put(166.5,102.6){\circle*{5}}
\put(133,122.3){\circle*{5}}
\put(133,83.7){\circle*{5}}
\put(133,83.7){\circle{7}}}
\put(0,116){\put(200,200){\circle*{5}}
\put(200,161){\circle*{5}}
\put(200,122.6){\circle*{5}}
\put(166.5,180.0){\circle*{5}}
\put(133,161){\circle*{5}}
\put(166.5,141.3){\circle*{5}}
\put(166.5,102.6){\circle*{5}}
\put(133,122.3){\circle*{5}}
\put(133,83.7){\circle*{5}}
\put(133,83.7){\circle{7}}}
\end{picture}

\end{document}